\documentclass[12pt]{amsart}
\usepackage[english]{babel}
\usepackage{amsfonts,amssymb,latexsym,amscd}

\theoremstyle{plain}
\newtheorem{theorem}{Theorem}

\newtheorem{proposition}{Proposition}
\newtheorem{corollary}{Corollary}
\theoremstyle{definition}
\newtheorem{definition}{Definition}
\theoremstyle{remark}
\newtheorem{remark}{Remark}
\newtheorem{example}{Example}

\usepackage{hyperref}

\usepackage{comment}

\usepackage{graphicx}
\usepackage[all]{xy}

\DeclareMathOperator{\sd}{sd}
\DeclareMathOperator{\ord}{ord}

\begin{document}

\title{Shape dimension of maps}

\author{Pavel S. Gevorgyan}
\address{Moscow Pedagogical State University}
\email{pgev@yandex.ru}

\author{I. Pop}
\address{Faculty of Mathematics, "Al. I. Cuza" University,
700505-Iasi, Romania}
\email{ioanpop@uaic.ro}

\begin{abstract}
In this paper we define a numerical shape invariant of a continuous map called shape dimension of a map, which generalizes the shape dimension of a topological space. Some basic properties and applications of this invariant are given. The question of raising the shape dimension by shape finite-dimensional maps is solved. An example of a shape finite-dimensional surjective map between shape infinite-dimensional spaces is given.
\end{abstract}



\maketitle

\section{Introduction}

This paper continues the study of shape invariant properties of maps  initiated by the authors in \cite{Gev-Pop-2, Gev-Pop-1}. We define a numerical shape invariant of a continuous map $f:X\rightarrow Y$ called shape dimension of a map and denoted by $\sd f$,  which generalizes the shape dimension of a topological space. 

We denote the shape category of topological spaces by $\text{Sh}_{\text{(HTop,\,HPol)}}$, where HTop is the homotopy category of spaces and  HPol is the homotopy category of polyhedra. Recall that the objects of HTop are all topological spaces and the morphisms are the homotopy classes $[f]$ of continuous maps $f:X\to Y$, which are called H-\emph{maps}. By  HPol we mean the full subcategory of HTop whose objects are all spaces having the homotopy type of a polyhedron. The full subcategory of HTop whose objects are all spaces $X$, which are homotopy dominated by $n$-dimensional  polyhedron $P$, is denoted by HPol$_n$.

Recall that for a simplicial complex $K$, $\dim K\leq n$ means that every simplex of $K$ has dimension $\leq n$. Then a covering $\mathcal{U}$ of a space $X$ is of order $\ord \mathcal{U}\leq n$ provided the nerve $N(\mathcal{U})$ has dimension $\leq n$, i.e., whenever $U_0,\ldots,U_k$ are members of $\mathcal{U}$, and $U_0\cap \ldots\cap U_k\neq \emptyset$, then $k\leq n$. By this, a topological space $X$ has covering dimension $\dim X\leq n$ provided every normal covering $\mathcal{U}$ of $X$ admits a normal refinement $\mathcal{V}$ with $\ord \mathcal{V}\leq n$. If $P$ is a polyhedron $P=|K|$, then $\dim X\leq n$ if and only if $\dim K\leq n$. The covering dimension $\dim X$ has a number of interesting properties but obviously it is not a homotopy invariant. This inconvenience has been solved for compact metric spaces by K. Borsuk \cite{Borsuk, Borsuk2} under the name of \textsl{fundamental dimension}. This notion was then studied in depth by S. Nowak \cite{Nowak}. Finally for arbitrary topological spaces a convenient notion of dimension was given by J. Dydak \cite{Dydak} as follows. 

In the shape theory based on inverse systems (see   \cite{Mard-Segal}, Chapter II, \S 1) a system $\mathbf{X}=(X_\lambda,p_{\lambda\lambda'},\Lambda)\in $ pro-HPol is said to be of dimension $\dim \mathbf{X}\leq n $ provided $\mathbf{X}\in $ pro-HPol$_n$, i.e., each $X_\lambda$ is homotopy dominated by a polyhedron $P$ with  $\dim P\leq n$. Then a space $X$ has \textsl{shape dimension} $\sd X\leq n$ provided $X$ admits an HPol$_n$-expansion $\mathbf{p}:X\rightarrow \mathbf{X}$, i.e., an H-Pol-expansion with $\dim \mathbf{X}\leq n$. This property of topological spaces is a shape invariant and thus a homotopy invariant.

In the second section of this paper we define the dimension of a morphism of inverse systems and of a pro-morphism in the category HPol, which generalizes the dimension of an inverse system in HPol (Definition \ref{def.2.1}), and then the shape dimension of a continuous map is defined (Definition \ref{def.2.12}). This notion generalizes the shape dimension of a topological space.
It is proved that a topological space $X$ has the shape dimension $\leq n$ if and only if the identity map $1_X$ has the shape dimension $\leq n$ (Theorem \ref{cor.2.13}).
If the topological space $X$ or $Y$ has the shape dimension $\leq n$  then any shape morphism between these spaces has the shape dimension $\leq n$ as well (Theorem \ref{thm.2.16}). The converse is true for a shape morphism which has a left or right inverse (Theorems \ref{prop_2} and \ref{th_1}). In particular, if $F:X\rightarrow Y$ is a shape isomorphism of topological spaces and has the shape dimension $\sd F=n$ then $X$ and $Y$ are shape $n$-dimensional spaces (Theorem \ref{cor.2.15}). 
When $\sd X=\infty$ then the shape dimension of the identity map $1_X$ is  infinite. 
Example \ref{ex.2.20} shows that there are shape finite-dimensional surjective maps between shape infinite-dimensional spaces. This means that the shape dimension of a map $f:X\to Y$ in general does not depend on the shape dimensions of spaces $X$ and $Y$. Consequently, the notion of the shape dimension of maps is of particular interest in the class of all shape infinite-dimensional spaces.

In the third section it is proved that the notion of the shape dimension of map can defined by means of level morphisms of inverse systems in the category HPol (Theorem \ref{th_3}). A characterization of the shape dimension of maps is given in Theorem \ref{th_2}. Some applications of shape dimension of maps to homology pro-groups and \v{C}ech homology are established (Theorem \ref{prop.3.4} and Corollary \ref{cor.3.5}).

\section{Shape dimension of a map: definitions, properties, examples}

It is known \cite{Mard-Segal} ( Ch.II, \S 1.1, Th. 2, p. 96) that the relation $\sd X\leq n$, for the shape dimension of a space $X$, can be characterized by the following condition:

D3. \textsl{If} $\mathbf{p}:X\rightarrow \mathbf{X}=(X_\lambda,p_{\lambda\lambda'},\Lambda)$ \textsl{is an} HPol-\textsl{expansion, then ever}y $\lambda\in \Lambda$ \textsl{admits} \textsl{a} $\lambda'\geq \lambda$ \textsl{such that} $p_{\lambda\lambda'}$ \textsl{factors in} HPol \textsl{through a } \textsl{polyhedron} $P$ \textsl{with} $\dim P\leq n$.

If $(f_\mu,\phi):\mathbf{X}=(X_\lambda,p_{\lambda\lambda'},\Lambda)\rightarrow \mathbf{Y}=(Y_\mu,q_{\mu\mu'},M)$ is a morphism of inverse systems in HPol, and if $(\lambda,\mu)\in \Lambda\times M$, with $\lambda\geq \phi(\mu)$, denote
\[
f_{\mu\lambda}:=f_\mu\circ p_{\phi(\mu)\lambda}.
\]

\begin{definition}\label{def.2.1}We say that a morphism of inverse systems $(f_\mu,\phi):\mathbf{X}=(X_\lambda,p_{\lambda\lambda'},\Lambda)\rightarrow \mathbf{Y}=(Y_\mu,q_{\mu\mu'},M)$ in the category HPol has dimension $\dim (f_\mu,\phi)\leq n$ if every $\mu\in M$ admits a $\lambda\geq \phi(\mu)$ such that the H-map $f_{\mu\lambda}:X_\lambda\rightarrow Y_\mu$ factors in HPol through a polyhedron $P$ with $\dim P\leq n$, i.e., there are H-maps $u:X_\lambda\rightarrow P$, $v:P\rightarrow Y_\mu$ such that $f_{\mu\lambda}=v\circ u$
\[
 \xymatrix{
X_\lambda \ar[rr]^{f_{\mu\lambda}}\ar[dr]_{u} & & Y_\mu\\
 & P \ar[ur]_{v} &
 }
\]
\end{definition}

\begin{remark}\label{rem.2.2}
It is obvious if $f_{\mu\lambda}$ factors through $P$, and $\lambda'\geq \lambda$, then $f_{\mu\lambda'}$ also factors through $P$.
Indeed, if $f_{\mu\lambda}=v\circ u$, then $f_{\mu\lambda'}=f_{\mu\lambda}\circ p_{\lambda\lambda'}=v\circ u'$, with $u'=u\circ p_{\lambda\lambda}:X_{\lambda'}\rightarrow P$.
\end{remark}

\begin{proposition}\label{prop_1}
Let $(f_\mu,\phi):\mathbf{X}=(X_\lambda,p_{\lambda,\lambda'},\Lambda)\rightarrow \mathbf{Y}=(Y_\mu,q_{\mu\mu'},M)$ be a morphsim of inverse systems in the category HPol. If $\dim \mathbf{X}\leq n$ or $\dim \mathbf{Y}\leq n$ then $\dim (f_\mu,\phi)\leq n$.
\end{proposition}

\begin{proof}
If $\dim \mathbf{X}\leq n$ and $\mu\in M$, then for $\phi(\mu)\in \Lambda$ there exist $\lambda\in \Lambda, \ \lambda\geq \phi(\mu)$, a polyhedron $P$ with $\dim P\leq n$ and two morphisms $u:X_\lambda\rightarrow P$, $v:P\rightarrow X_{\phi(\mu)}$ in HPol such that $p_{\phi(\mu)\lambda}=v\circ u$. Then $f_{\mu\lambda}=f_\mu\circ p_{\phi(\mu)\lambda}=(f_\mu\circ v)\circ u$, with $f_\mu\circ v:P\rightarrow Y_\mu$.

If $\dim \mathbf{Y}\leq n$ and $\mu\in M$, then there exist $\mu'\in M,\ \mu'\geq \mu$, a polyhedron $P'$ with $\dim P'\leq n$, and two morphisms $u':Y_{\mu'}\rightarrow P'$, $v':P'\rightarrow Y_\mu$ in HPol such that $q_{\mu\mu'}=v'\circ u'$. Consider $\lambda\in \Lambda $ satisfying
$\lambda \geq \phi(\mu),\phi(\mu')$ and $f_{\mu\lambda}=q_{\mu\mu'}\circ f_{\mu'\lambda}$. Then $f_{\mu\lambda}=v'\circ (u'\circ f_{\mu'\lambda})$,
with $u'\circ f_{\mu'\lambda}:X_\lambda\rightarrow P'$. Therefore in both cases Definition \ref{def.2.1} is satisfied.
\end{proof}

\begin{example}\label{ex.2.5}
It follows from Proposition \ref{prop_1} that if $\mathbf{X}=(X_\lambda,p_{\lambda\lambda'},\Lambda)$ is an inverse system in HPol with $\dim \mathbf{X}\leq n$ then $\dim 1_\mathbf{X}\leq n$. 
\end{example}

\begin{example}\label{ex.2.4}
Let $(f_\mu,\phi):(X_\lambda,p_{\lambda\lambda'},\Lambda)\rightarrow (Y_\mu,q_{\mu\mu'},M)$ be a a morphism of inverse systems in the category Pol of polyhedra. For $\mu\in M$ denote by $Y^n_\mu$ the carrier of the n-skeleton of a triangulation of the polyhedron $Y_\mu$. Suppose that for every $\mu\in M$ we have $f_\mu(X_{\phi(\mu)})\subseteq Y^n_\mu$. Then the morphism in the category HPol, $([f_\mu],\phi):\mathbf{X}=(X_\lambda,[p_{\lambda\lambda'}],\Lambda)\rightarrow \mathbf{Y}=(Y_\mu,[q_{\mu\mu'}],M)$ has $\dim ([f_\mu],\phi)\leq n$.
\end{example}

\begin{proposition}\label{prop.2.6}
If $(f_\mu,\phi),(f'_\mu,\phi'):\mathbf{X}=(X_\lambda,p_{\lambda\lambda'},\Lambda)\rightarrow \mathbf{Y}=(Y_\mu, q_{\mu\mu'}, M)$ are two equivalent morphisms in the category HPol, $(f_\mu,\phi)\sim (f'_\mu,\phi')$, and if $\dim (f_\mu,\phi)\leq n$, then $\dim (f'_\mu,\phi')\leq n$. 
\end{proposition}

\begin{proof}
For $\mu\in M$, there exists $\lambda\in \Lambda$, $\lambda\geq \phi(\mu),\phi'(\mu)$ such that $f_{\mu\lambda}=f'_{\mu\lambda}$. Now by Remark \ref{rem.2.2}, we can suppose that $f_{\mu\lambda}=v\circ u$, for some H-maps $u:X_\lambda\rightarrow P$, $v:P\rightarrow Y_\mu$, with $\dim P\leq n$. Then $f'_{\mu\lambda}=v\circ u$, so $\dim (f'_\mu,\phi')\leq n$.
\end{proof}

Based on Proposition \ref{prop.2.6}, we can give the following definition.

\begin{definition}\label{def.2.7}
A pro-morphism $\mathbf{f}:\mathbf{X}=(X_\lambda,p_{\lambda\lambda'},\Lambda)\rightarrow \mathbf{Y}=(Y_\mu,q_{\mu\mu'},M)$, in
pro-HPol, has the dimension $\dim \mathbf{f}\leq n$ provided $\mathbf{f}$ admits a represenation by a morphism $(f_\mu,\phi):\mathbf{X}\rightarrow \mathbf{Y}$ with $\dim (f_\mu,\phi)\leq n$.
\end{definition}

\begin{proposition}\label{prop.2.8}Let $\mathbf{f}:\mathbf{X}\rightarrow \mathbf{Y}$ and $\mathbf{g}:\mathbf{Y}\rightarrow \mathbf{Z}$ be pro-morphisms in HPol.
If $\dim \mathbf{f}\leq n$, then $\dim (\mathbf{g}\circ \mathbf{f})\leq n$.
\end{proposition}
\begin{proof}Suppose that $\mathbf{f}$ and $\mathbf{g}$ are represented  by the morphisms
$(f_\mu,\phi): \mathbf{X}=(X_\lambda,p_{\lambda\lambda'},\Lambda)\rightarrow \mathbf{Y}=(Y_\mu,q_{\mu\mu'},M)$ and respectively
$(g_\nu,\psi):\mathbf{Y}\rightarrow \mathbf{Z}=(Z_\nu,r_{\nu\nu'},N)$. Then $\mathbf{h}:=\mathbf{g}\circ \mathbf{f}:\mathbf{X}\rightarrow \mathbf{Z}$ is represented by the morphism $(h_\nu,\chi):\mathbf{X}\rightarrow \mathbf{Z}$, with $\chi=\phi\circ \psi$ and $h_\nu=g_\nu\circ f_{\psi(\nu)}$.

If $\nu\in N$, then for $\psi(\nu)\in M$ there exist $\lambda\in \Lambda, \lambda\geq \phi(\psi(\nu))=\chi(\nu)$ and some H-maps $u:X_\lambda\rightarrow P$, $v:P\rightarrow Y_{\psi(\nu)}$, with $\dim P\leq n$, such that
\begin{equation}\label{eq-2.1}
f_{\psi(\nu)\lambda}=v\circ u.
\end{equation}
But $h_{\nu\lambda}=h_\nu\circ p_{\chi(\nu)\lambda}=g_\nu\circ f_{\psi(\nu)}\circ p_{\psi(\phi(\nu))\lambda}=g_\nu\circ f_{\psi(\nu)\lambda}$. So by the relation \eqref{eq-2.1} we have $h_{\nu\lambda}=(g_\nu\circ v)\circ u$, with $g_\nu\circ v:P\rightarrow Z_\nu$. And this shows that $\dim \mathbf{h}\leq n$.
\end{proof}

\begin{proposition}\label{prop.2.9} Let $\mathbf{f}:\mathbf{X}\rightarrow \mathbf{Y}$ and $\mathbf{g}:\mathbf{Y}\rightarrow \mathbf{Z}$ be pro-morphisms in the
category HPol. If $\dim \mathbf{g}\leq n$, then $\dim (\mathbf{g}\circ \mathbf{f})\leq n$.
\end{proposition}

\begin{proof}
We use the same notations as in the above proof. Now for $\nu\in N$, there exist $\mu\geq \psi(\nu)$, a polyhedron $P'$ with $\dim P'\leq n$ and some H-maps
$ u':Y_\mu\rightarrow P'$, $v':P'\rightarrow Z_\nu$, satisfying
\begin{equation}
g_{\nu\mu}=v'\circ u'.
\end{equation}
Now conider $\lambda\in \Lambda$, $\lambda\geq \phi(\mu),\lambda\geq \phi(\psi(\nu))$, such that
$q_{\psi(\nu)\mu}\circ f_\mu\circ p_{\phi(\mu)\lambda}=f_{\psi(\nu)}\circ p_{\phi(\psi(\nu))\lambda}$, that is $q_{\psi(\nu)\mu}\circ f_{\mu\lambda}=f_{\psi(\nu)\lambda}$. By this and by the relation $h_{\nu\lambda}=g_\nu\circ f_{\lambda\psi(\nu)}$, above deduced, we have
$h_{\nu\lambda}=g_\nu\circ q_{\psi(\nu)\mu} \circ f_{\mu\lambda}=g_{\nu\mu}\circ f_{\mu\lambda}=(v'\circ u')\circ f_{\mu\lambda}=
v'\circ (u'\circ f_{\mu\lambda})$, with $u'\circ f_{\mu\lambda}:X_\lambda\rightarrow P$, by which we conclude that $\dim (\mathbf{g}\circ \mathbf{f})\leq n$.
\end{proof}

\begin{proposition}\label{prop.2.10}
Suppose a commutative diagram is given in pro-HPol
\[
\xymatrix{\mathbf{X}\ar[d]_{\mathbf{\alpha}} \ar[r]^{\mathbf{f}} & \mathbf{Y}\ar[d]^{\mathbf{\beta}}\\
\mathbf{X}'\ar[r]_{\mathbf{f}'} & \mathbf{Y}'}
\]
with $\mathbf{\alpha}$ and $\mathbf{\beta}$ isomorphisms. If $\dim \mathbf{f}\leq n$, then $\dim \mathbf{f}'\leq n$.
\end{proposition}

\begin{proof}
Because $\dim \mathbf{f}\leq n$, by Proposition \ref{prop.2.9} $\dim (\mathbf{\beta} \circ \mathbf{f})\leq n$, so $\dim (\mathbf{f}'\circ \mathbf{\alpha})\leq n$. But $\mathbf{f}'=(\mathbf{f}'\circ \mathbf{\alpha})\circ {\mathbf{\alpha}}^{-1}$, and then by Proposition \ref{prop.2.8},
we obtain that $\dim \mathbf{f}'\leq n$.
\end{proof}

\begin{example}\label{ex.2.11}
Let $\mathbf{f}:\mathbf{X}\rightarrow \mathbf{Y}=(Y_\mu,q_{\mu\mu'},M)$ be a pro-morphism of the category HPol. Suppose that $(M',\leq )$ is a cofinal subset of $(M,\leq )$. Denote by $\mathbf{Y}'$ the subsistem of $\mathbf{Y} $ defined by $(M',\leq)$, i.e., $\mathbf{Y}'=(Y_{\mu'},q_{\mu'\mu''},M')$ and let $\mathbf{i}:\mathbf{Y}\rightarrow \mathbf{Y}'$ be the restriction morphism and $\mathbf{f}':\mathbf{X}\rightarrow \mathbf{Y}'$ the pro-morphism induced by $f$ and the inclusion $M'\hookrightarrow M$. Then $\mathbf{f}'=\mathbf{i}\circ \mathbf{f}$ and because  $\mathbf{i}$ is an isomorphism in pro-HPol (\cite{Mard-Segal}, Ch.I, \S 1.1, Th. 1), we have that $\dim \mathbf{f}= \dim \mathbf{f}'$.
\end{example}

Because of Proposition \ref{prop.2.10} we can give the following definition.

\begin{definition}\label{def.2.12} 
A shape morphism $F:X\rightarrow Y$ between topological spaces has \emph{shape dimension} $\sd  F\leq n$, $n\geq 0$, provided $F$ admits a representation $\mathbf{f}:\mathbf{X}\rightarrow \mathbf{Y}$ with $\dim \mathbf{f}\leq n$.

In particular, an H-map $f:X\rightarrow Y$ has shape dimension $\sd f\leq n$ if its shape image by the shape functor $S:HTop\rightarrow Sh$ has $sd(S(f))\leq n$. If $f\in \text{Top}(X,Y)$, we write $\sd f\leq n$ if $\sd ([f])\leq n$.

We put $\sd F=n$ (or  $\sd f=n$) provided $n$ is the least $m$ for which $\sd F\leq m$ (resp. $\sd f\leq m$).
\end{definition}

If we combine the above mentioned characterization D3, Example \ref{ex.2.5}, Definitions \ref{def.2.1}, \ref{def.2.7} and \ref{def.2.12}, we obtain the following theorem.

\begin{theorem}\label{cor.2.13}
A topological space $X$ has shape dimension $\sd X\leq n$ if and only if the identity map of $X$ has shape dimension $\sd (1_X)\leq n$.
\end{theorem}

By Propositions \ref{prop.2.8} and \ref{prop.2.9} we have the following corollary.

\begin{corollary}\label{cor.2.14}
Let $F:X\rightarrow Y$, $G:Y\rightarrow Z$ be two shape morphisms of topological sapces. If $\sd F\leq n$ or $\sd G\leq n$, then $\sd (G\circ F)\leq n$.
\end{corollary}

\begin{theorem}\label{thm.2.16}
Let $F:X\rightarrow Y$ be a shape morphism of topological spaces. If $\sd X\leq n$ or $\sd Y\leq n$ then $\sd F\leq n$.
\end{theorem}

\begin{proof}
Suppose that $\mathbf{p}=(p_\lambda):X\rightarrow \mathbf{X}=(X_\lambda,p_{\lambda\lambda'},\Lambda)$  and $\mathbf{q}=(q_\mu):Y\rightarrow
\mathbf{Y}=(Y_\mu,q_{\mu\mu'},M)$ are HPol-expansions for $X$ and $Y$ respectively, and let $\mathbf{f}=(f_\mu,\phi):\mathbf{X}\rightarrow \mathbf{Y}$ determine a shape morphism $F:X\to Y$.

Since $\sd X\leq n$ ($\sd Y\leq n$) means $\dim \mathbf{X}\leq n$ ($\dim \mathbf{Y}\leq n$) and $\sd F\leq n$ means $\dim (f_\mu,\phi)\leq n$, it only remains to apply Proposition \ref{prop_1}.
\end{proof}

\begin{corollary}\label{cor.2.17}
For a continuous map $f:X\rightarrow Y$, $\sd f\leq \min (\sd X, \sd Y)$.
\end{corollary}

Since for a topological space $X$ one has the inequality $\sd X\leq \dim X$ (\cite{Mard-Segal}, Ch.II, \S 1.1, Th. 3), by Theorem \ref{thm.2.16}, we obtain the following corollary.

\begin{corollary}\label{cor.2.19}
Let $f:X\rightarrow Y$ be a continuous map. If $\dim X\leq n$  or $\dim Y\leq n$, then $\sd f\leq n$.
\end{corollary}

\begin{corollary}\label{cor.2.23}
If a shape morphism $F:X\rightarrow Y$ of topological spaces factors in the shape category $\text{Sh}_{\text{(HTop,\,HPol)}}$ by a topological space $Z$ of shape dimension $\sd Z\leq n$ then $\sd F\leq n$.

In particular, if a continuous map $f:X\rightarrow Y$ factors in the category Top by a topological space $Z$ of shape dimension $\sd Z\leq n$ then $\sd f\leq n$.
\end{corollary}

\begin{proof}
Suppose that $F=H\circ G$, for shape morphisms $G:X\rightarrow Z$ and $H:Z\rightarrow Y$, where $\sd  Z\leq n$.  It follows from Theorem \ref{thm.2.16} that $\sd G\leq n$ and $\sd  H\leq n$.  Hence, by Corollary \ref{cor.2.14} $\sd F\leq n$.
\end{proof}

\begin{example}\label{ex.2.24}
Using the last corollary one can easily see that if $f:X\rightarrow Y$ is nullhomotopic then $\sd f=0$.
\end{example}

\begin{example}\label{ex.2.20} 
Let $Q$ be the Hilbert cube,  seen as the product of countably infinitely many copies of the unit interval $[0, 1]$,  and $\mathbb{T}^\omega$ the infinite-dimensional torus, as the product of countably infinitely many copies of the unit circle $S^1$. Consider the maps $g:Q\rightarrow \mathbb{T}^\omega$, defined  as the product of countably infinitely many copies of the usual exponential map $exp:[0,1]\rightarrow S^1, exp(t)=e^{2\pi i t}$, and, for $n\geq 0$,  $h_n:\mathbb{T}^\omega\rightarrow \mathbb{T}^\omega$ the projection on the $n$-dimensional torus $h_n((z_k)_k)=(z_1, \ldots ,z_n,1,1, \ldots)$ ($h_0$ is the constant map $(z_k)_k\rightarrow (1,1,\ldots))$. Then if $X$ is the topological sum $X:=Q\oplus \mathbb{T}^\omega$, define the map $\underline{f:X\rightarrow \mathbb{T}^\omega}$ by  $f|Q=g$ and $f|\mathbb{T}^\omega=h_n$. Now we have $\underline{\sd X=\infty}$, $\underline{\sd  \mathbb{T}^\omega=\infty}$. Since $Q$ is contractible it follows that $\sd Q=0$, so $\sd g=0$ (see Theorem \ref{thm.2.16}). Then  by Example \ref{ex.2.4} we have $\sd h=n$. In this way we obtain $\underline{\sd f=n}$.
\end{example}

\begin{remark}\label{rem.2.21}
Example \ref{ex.2.20}, the underlined relations together with the fact that the map $f$ is surjective  show that the notion of shape dimension of a map  $f:X\rightarrow Y$ is  consistent, independent of shape dimensions of $X$ and $Y$. This enables us to give some important new shape properties of $f$.
\end{remark}

The next two theorems show that the converse of Theorem \ref{thm.2.16} is true for a shape morphism which has a left or right inverse.

\begin{theorem}\label{prop_2}
Let $F:X\rightarrow Y$ be a shape morphism of topological spaces which has a left inverse. If $\sd F\leq n$
then $\sd X \leq n$.
\end{theorem}

\begin{proof}
Let $G:Y\to X$ be a left inverse of $F:X\to Y$, i.e., $G\circ F=1_X$. If $\sd F\leq n$ then by Corollary \ref{cor.2.14} $\sd (G\circ F)\leq n$, i.e., $\sd 1_{X}\leq n$. Therefore, by Theorem \ref{cor.2.13}, $\sd X \leq n$.
\end{proof}

The following theorem is proved in the same way.

\begin{theorem}\label{th_1}
Let $F:X\rightarrow Y$ be a shape morphism of topological spaces which has a right inverse. If $\sd F\leq n$
then $\sd Y \leq n$.
\end{theorem}

By Theorems  \ref{thm.2.16}, \ref{prop_2} and \ref{th_1} one can obtain the following result.

\begin{theorem}\label{cor.2.15}
Let $F:X\rightarrow Y$ be a shape isomorphism of topological spaces. If $\sd F=n$ then $\sd X=\sd Y=n$.
\end{theorem}

Note that Theorem \ref{cor.2.15} gives a new proof of the shape invariance of the shape dimension of a space.

\begin{definition}\label{def.2.25}We say that a morphism of pointed  inverse systems $((f_\mu,\ast),\phi):(\mathbf{X},\ast)=((X_\lambda,\ast)),p_{\lambda\lambda'},\Lambda)\rightarrow (\mathbf{Y},\ast)=((Y_\mu,\ast),q_{\mu\mu'},M)$ in the category HPol$_\ast$  has dimension $\dim ((f_\mu,\ast),\phi)\leq n$ if every $\mu\in M$ admits a $\lambda\geq \phi(\mu)$ such that the pointed H-map $(f_{\mu\lambda},\ast):(X_\lambda,\ast)\rightarrow (Y_\mu,\ast)$ factors in HPol$_\ast$ through a pointed polyhedron $(P,\ast)$ with $\dim P\leq n$, i.e., there are pointed H-maps $(u,\ast):(X_\lambda,\ast)\rightarrow (P,\ast)$, $(v,\ast):(P,\ast)\rightarrow (Y_\mu,\ast)$ such that $(f_{\mu\lambda},\ast)=(v,\ast)\circ (u,\ast)$.
\end{definition}
\begin{remark}\label{rem.2.26}With this definition all other definitions and results from the unpointed case remain valid in the pointed case.
\end{remark}

\section{More properties and some applications}

\begin{theorem}\label{thm.3.1} Let $\mathbf{f}:\mathbf{X}=(X_\lambda,p_{\lambda\lambda'},\Lambda)\rightarrow \mathbf{Y}=(Y_\mu,q_{\mu\mu'},M)$ be a morphism in the cagtegory pro-HPol. If $\dim \mathbf{f}\leq n$ then there exists a representation of $\mathbf{f}$  by a morphism $(f'_\mu,\phi'):\mathbf{X}\rightarrow \mathbf{Y}$ with the property that for every $\mu\in M$, the morphism $f'_\mu:X_{\phi'(\mu)}\rightarrow Y_\mu$ factors in HPol through a polyhedron $P$ with $\dim P\leq n$.
\end{theorem}

\begin{proof}
Suppose that $\mathbf{f}$ is given by an arbitrary morphism $(f_\mu,\phi):\mathbf{X}\rightarrow \mathbf{Y}$. Then by hypothesis, every $\mu \in M$
admits an index $\lambda_\mu\in \Lambda, \lambda_\mu\geq \phi(\mu)$, a polyhedron $P$ with $\dim P\leq n$ and two H-maps $u:X_{\lambda_\mu}\rightarrow P$, $v:P\rightarrow Y_\mu$ such that
\[
f_{\mu\lambda_\mu}=v\circ u.
\]
Then we define a new function $\phi':M\rightarrow \Lambda$ by
\[
\phi'(\mu)=\lambda_\mu,
\]
and consider the H-maps $f'_\mu:X_{\phi'(\mu)}\rightarrow Y_\mu$ given by
\[
f'_\mu=f_{\mu\lambda_\mu}.
\]
Now we prove that the set of H-maps $\{f'_\mu;\ \mu\in M\}$ defines a morphism of inverse systems $(f'_\mu,\phi'):\mathbf{X}\rightarrow \mathbf{Y}$.
Indeed, if $\mu,\mu'\in M$ with $\mu'\geq \mu$, then because $(f_\mu,\phi)$ is a morphism of inverse systems, there exists $\overline{\lambda}\geq \phi(\mu),\phi(\mu')$ such that
\[
q_{\mu\mu'}\circ f_{\mu'}\circ p_{\phi(\mu')\overline{\lambda}}=f_\mu\circ p_{\phi(\mu)\overline{\lambda}}.
\]
Now consider an index $\widetilde{\lambda}\in \Lambda$, $\widetilde{\lambda}\geq \lambda_\mu,\lambda_{\mu'},\overline{\lambda}$. Then we have:
\begin{multline*}
q_{\mu\mu'}\circ f'_{\mu'}\circ p_{\phi'(\mu')\widetilde{\lambda}}=q_{\mu\mu'}\circ f_{\mu'}\circ p_{\phi(\mu')\lambda_{\mu'}}\circ
p_{\lambda_{\mu'}\widetilde{\lambda}}=\\
q_{\mu\mu'}\circ f_{\mu'}\circ (p_{\phi(\mu')\phi'(\mu')}\circ p_{\phi'(\mu')\widetilde{\lambda}})=
q_{\mu\mu'}\circ f_{\mu'}\circ p_{\phi(\mu')\overline{\lambda}}\circ p_{\overline{\lambda} \widetilde{\lambda}}=f_\mu\circ p_{\phi(\mu)\overline{\lambda}}\circ p_{\overline{\lambda} \widetilde{\lambda}}=\\
(f_\mu\circ p_{\phi(\mu)\phi'(\mu)})\circ p_{\phi'(\mu)\widetilde{\lambda}}=f'_\mu\circ p_{\phi'(\mu)\widetilde{\lambda}}.
\end{multline*} 
Therefore $(f'_\mu,\phi')$ is a morphism of inverse systems. 
In addition, obviously $(f'_\mu,\phi')\sim (f_\mu,\phi)$ because for every $\mu\in M$, $f_\mu\circ p_{\phi(\mu)\phi'(\mu)}=f'_\mu\circ p_{\phi'(\mu)\phi'(\mu)}$. So, $[(f'_\mu,\phi')]=\mathbf{f}$.

Finally, for every $\mu\in M$, we have $f'_\mu=v\circ u$, with $u:X_{\phi'(\mu)}\rightarrow P$, $v:P\rightarrow Y_\mu$, and $\dim P\leq n$, which ends the proof.
\end{proof}

Using Theorem \ref{thm.3.1} and \cite{Mard-Segal} (Ch.I, \S 1.3, Theorem 2) we obtain the following result.

\begin{theorem}\label{th_3}
If $F:X\rightarrow Y$ is a shape morphism of topological spaces with $\sd F\leq n$ then there exist a cofinite directed ordered set $(N,\leq )$ and a level morphism of inverse systems in the category HPol $(f_\nu):\mathbf{X}=(X_\nu,p_{\nu\nu'},N)\rightarrow \mathbf{Y}=(Y_\nu,q_{\nu\nu'},N)$ defining a pro-morphism $\mathbf{f}:\mathbf{X}\rightarrow \mathbf{Y}$ which in turn is a representation of $F$, and such that for every $\nu\in N$ the morphism $f_\nu:X_\nu\rightarrow Y_\nu$ factors in HPol through a polyhedron $P$ with $\dim P\leq n$.
\end{theorem}

The following theorem characterizes the shape dimension of map. 

\begin{theorem}\label{th_2}
An H-map $f:X\rightarrow Y$ has shape dimension $\sd f\leq n$ if and only if for every H-map $h:Y\rightarrow P$ into a space $P\in$ HPol, the composition $h\circ f$  factors through a polyhedron $P'$ with $\dim P'\leq n$, i.e., there are H-maps $u:X\rightarrow P'$ and $v:P'\rightarrow P$ such that $h\circ f=v\circ u$.
\end{theorem}

\begin{proof}
Let  $\sd f\leq n$. Consider a morphism $\mathbf{f}:\mathbf{X}=(X_\lambda,p_{\lambda\lambda'},\Lambda)\rightarrow \mathbf{Y}=(Y_\mu,q_{\mu\mu'},M)$ in the cagtegory pro-HPol, induced by an H-map $f:X\rightarrow Y$. Let $(f_\mu, \phi)$ be a representative of $\mathbf{f}$.

For any  H-map $h:Y\rightarrow P$ into a space $P\in$ HPol there are a $\mu \in M$ and an H-map $g:Y_\mu \to P$  such that $h=g\circ q_\mu$. By assumption there exists a $\lambda \in \Lambda$, $\lambda\geq \phi(\mu)$ such that an H-map $f_{\mu \lambda}$ factors through a polyhedron $P'$ with $\dim P'\leq n$. Consider H-maps $u':X_\lambda\rightarrow P'$ and $v':P'\rightarrow Y_\mu$ such that $v'\circ u'=f_{\mu \lambda}$. Let us denote $u=u'\circ p_\lambda:X \to P'$ and $v=g\circ v':P' \to P$. Now it is readily seen that $h\circ f=v\circ u$, i.e., the composition $h\circ f$  factors through a polyhedron $P'$.

Conversely, assume that for every H-map $h:Y\rightarrow P$ the composition $h\circ f$ factors through a polyhedron $P'$ with $\dim P'\leq n$, i.e., $h\circ f=v\circ u$. 

Let $\mu\in M$ and consider the H-map $q_\mu\circ f:X\to Y_\mu$. By assumption there exist a polyhedron $P'$ with $\dim P'\leq n$ and H-maps $u:X\to P'$, $v:P'\to Y_\mu$ such that 
\begin{equation}\label{eq_4}
v\circ u = q_\mu \circ f.
\end{equation}
For the H-map $u:X\to P'$ there exist a $\lambda \in \Lambda$, $\lambda\geq \phi(\mu)$ and an H-map $g:X_\lambda\to P'$ such that
\begin{equation}\label{eq_5}
g\circ p_\lambda = u.
\end{equation}

By \ref{eq_4} and \ref{eq_5} it is not difficult to verify the equality
\begin{equation*}
(f_\mu\circ p_{\phi(\mu)\lambda})\circ p_\lambda= (v\circ g)\circ p_\lambda.
\end{equation*}
Hence, there exists a $\lambda'\in \Lambda$, $\lambda'\geq \lambda$, such that 
\begin{equation*}
(f_\mu\circ p_{\phi(\mu)\lambda})\circ p_{\lambda\lambda'}=(v\circ g)\circ p_{\lambda\lambda'},
\end{equation*}
and therefore $f_{\mu\lambda'}=v\circ u$, where $u=g\circ p_{\lambda\lambda'}$. Thus, $\sd f\leq n$.
\end{proof}

\begin{remark}\label{rem.3.7}
The condition from Theorem \ref{th_2} is analogous to the condition (D2) from \cite{Mard-Segal}, Ch. II, \S 1.1, Theorem 2.
\end{remark}

\begin{proposition}
Let $X$ be a topological space, $A\subseteq X$ a weak retract, $i:A\hookrightarrow X$ the inclusion map, and $r:X\rightarrow A$ the weak retraction. Then the following conditions are equivalent:

(i) $\sd A\leq n$, 

(ii) $\sd i \leq n$, 

(iii) $\sd r \leq n$.
\end{proposition}
\begin{proof}
Since $S([r])\circ S([i])=1_{S(A)}$, one can apply Theorems \ref{prop_2}, \ref{th_1} and  \ref{cor.2.13}.
\end{proof}

\begin{theorem}\label{prop.3.4}
Let $\mathbf{f}:\mathbf{X}\rightarrow \mathbf{Y}$ be a pro-morphism in the category HPol. If $\dim \mathbf{f}\leq n$, then for every Abelian group $G$ and an index $k>n$ the homology pro-morphism $H_k(\mathbf{f};G):H_k(\mathbf{X};G)\rightarrow H_k(\mathbf{Y};G)$ is a zero-morphism of pro-groups.
\end{theorem}

\begin{proof}
If $\mathbf{f}=(f_\mu,\phi):\mathbf{X}=(X_\lambda,p_{\lambda\lambda'},\Lambda)\rightarrow \mathbf{Y}=(Y_\mu, q_{\mu\mu'},M)$, then $H_k(\mathbf{X};G)=(H_k(X_\lambda;G),H_k(p_{\lambda\lambda'};G),\Lambda)$, $H_k(\mathbf{Y};G)=(H_k(Y_\mu;G), H_k(q_{\mu\mu'};G), M)$ and $H_k(\mathbf{f};G)=(H_k(f_\mu;G),\phi):H_k(\mathbf{X};G)\rightarrow H_k(\mathbf{Y};G)$.

By hypothesis every index $\mu\in M$ admits $\lambda\in \Lambda$, $\lambda\geq \phi(\mu)$, a polyhedron $P_\mu$ with $\dim P_\mu\leq n$ and two H-maps $u_\mu:X_\lambda\rightarrow P_\mu$, $v_\mu:P_\mu\rightarrow Y_\mu$ such that $f_{\mu\lambda}=v_\mu\circ u_\mu$. This implies $H_k(f_\mu\circ p_{\phi(\mu)\lambda};G)=H_k(v_\mu;G)\circ H_k(u_\mu;G)$, i.e., $H_k(f_\mu;G)\circ H_k(p_{\phi(\mu)\lambda};G)=H_k(v_\mu;G)\circ H_k(u_\mu;G)$. Then since $\dim P_\mu\leq n$, for $k>n$ we have that $H_k(v_\mu;G)$ is a zero-morphism, so $H_k(f_\mu;G)\circ H_k(p_{\phi(\mu)\lambda};G)$ is also a zero-morphism. Then by \cite{Mard-Segal}, Ch.II, \S 2.3, Theorem 7, we conclude that $H_k(\mathbf{f}; G)$ is a zero-pro-morphsim of groups.
\end{proof}

\begin{corollary}\label{cor.3.5}
Suppose that $F:X\rightarrow Y$ is a shape morphism of toplogical spaces with $\sd F\leq n$. Then for any Abelian group $G$ and for every index  $k>n$ the \v{C}ech homology homomorphism $\check{F}_k:\check{H}_k(X;G)\rightarrow \check{H}_k(Y;G)$ is null.
\end{corollary}

\end{document}